\documentclass{amsart}

\usepackage{ mathrsfs }

\newtheorem{thm}{Theorem}
\newtheorem*{thm*}{Theorem}

\newtheorem{lem}[thm]{Lemma}
\newtheorem{cor}[thm]{Corollary}
\newtheorem{conj}[thm]{Conjecture} 

\theoremstyle{definition}

\theoremstyle{remark}

\begin{document}

\title{Cops and Robbers on diameter two graphs}

\author{Zsolt Adam Wagner}

\begin{abstract}
In this short paper we study the game of \emph{Cops and Robbers}, played on the vertices of some fixed graph $G$ of order $n$. The minimum number of cops required to capture a robber is called the cop number of $G$. We show that the cop number of graphs of diameter 2 is at most $\sqrt{2n}$, improving a recent result of Lu and Peng by a constant factor. We conjecture that this bound is still not optimal, and obtain some partial results towards the optimal bound.
\end{abstract}


\maketitle

\section{Introduction}

The game of \emph{Cops and Robbers}, introduced independently by Nowakowski and Winkler \cite{nowwi} and Quillot \cite{qui}, is a perfect information game played on a fixed graph $G$. There are two players, a set of $k\geq 1$ cops and the robber. The cops begin the game by occupying any vertices of their choice (where more than one cop can be placed at a vertex). Then the robber chooses a vertex for himself. Afterwards the cops and the robber move in alternate rounds, with cops going first. At each step any cop or robber is allowed to move along an edge of $G$ or remain stationary. The cops win if at some time there is a cop at the same vertex as the robber; otherwise, i.e., if the robber can elude the cops indefinitely, the robber wins. The minimum number of cops for which there is a winning strategy, no matter how the robber plays, is called the \emph{cop number} of $G$, and is denoted by $c(G)$. We will assume that $G$ is connected and simple, because deleting multiple edges or loops does not affect the possible moves of the players, and the cop number of a disconnected graph equals the sum of the cop numbers for each component. We write $c(n)$ for the maximum of $c(G)$ amongst all $n$-vertex connected graphs.

Currently the best known upper bound on the cop number of general graphs is due to Scott, Sudakov \cite{sudakovscottpaper} and Lu, Peng \cite{lupeng}. They showed that  $c(n)\leq n 2^{-(1+o(1))\sqrt{\log n}}$. The best known open question in this area is Meyniel's conjecture - which first appeared in \cite{frankl} - stating that $c(n)=O(\sqrt{n})$. This conjecture has been almost completely established for random graphs by Bollob\'{a}s, Kun, Leader \cite{bollobaskunleader}, \L uczak, Pra\l at \cite{zigzag} and Pra\l at, Wormald \cite{pralatwormald}. For general graphs, even $c(n) = O(n^{1-\epsilon})$ is completely open. For a survey of results on the cop number, we refer to \cite{survey2,survey1}.

In a recent paper \cite{lupeng}, Lu and Peng considered graphs of diameter 2, and bipartite graphs of diameter 3. They showed using a random argument that for such graphs we have $c(G) \leq 2\sqrt{n}-1$, and hence proved a special case of Meyniel's conjecture. Their result is tight up to a constant factor: an infinite class of diameter $2$ graphs with $c(G)\approx \sqrt{n}/2$ is given by Bonato and Burgess in \cite{meynielextremal} (the polarity graphs). The aim of this paper is to give a shorter proof, with no randomness involved, that gives a slightly stronger result.

\begin{thm}\label{main}
Let $G$ be a connected graph of diameter $2$, or a connected bipartite graph of diameter $3$, of order $n$. Then $c(G)\leq  \sqrt{2n}$.
\end{thm}

We will conjecture that the correct upper bound should be $\sqrt{n}$ for every graph, not just for diameter $2$ graphs, and provide some evidence in favour of the conjectured bound.

\section{Proof of the main result}
We first present the proof of Theorem \ref{main} for the diameter $2$ case. The following lemma was proved in \cite{lupeng}:

\begin{lem} \label{lem:degen}Let $k>0$ be an integer,  $G$ be a graph of diameter $2$, and let $H$ be a subgraph of $G$, such that the maximum degree of $H$ is at most $k$. Suppose the robber is restricted to move on the edges of $H$, while the cops can move on $G$ as usual. Then $k$ cops can catch the robber. 
\end{lem}

\begin{proof}
 Suppose a robber moves to a vertex $x$ with neighbours $v_1, v_2, \ldots, v_l$ with $l\leq k$. Let the cops be $c_1, c_2, \ldots, c_k$, and for each $i=1 \ldots l$, move $c_i$ to a neighbour of $v_i$. The robber is caught in at most two more rounds. 
\end{proof}

Given a graph $G$ and a subgraph $H \subseteq G$, write $c_G(H)$ for the number of cops needed to catch a robber who is restricted to move on $H$. (The cops are still allowed to move on all the edges of $G$, only the robber has restricted movement.) Also, given a graph $G$ and $m\leq |V(G)|$, define $c_G (m)=\max \{c_G (H) : H\subseteq G$, and $|V(H)|=m \}$.

\begin{proof}[Proof of Theorem \ref{main}] By induction on $m$, we will prove that $c_G(m)\leq \lfloor \sqrt{2m}\rfloor$ for every $m\leq n$. As $c_G(n)=c(G)$, this will imply the theorem. As $c_G(1)=c_G(2)=c_G(3)=1$, this holds for $m=1, 2, 3$.

Let $H \subseteq G$ of order $m\geq 4$. If every vertex of $H$ has degree at most $\lfloor \sqrt{2m}\rfloor$ then $c_G(H)\leq \lfloor \sqrt{2m}\rfloor$ by Lemma ~\ref{lem:degen}. Othervise, let $v$ be a vertex in $H$ of degree more than $\lfloor \sqrt{2m}\rfloor$. Put a stationary cop on $v$ to guard its neighbourhood.  Remove $v$ and its neighbourhood from $H$ to obtain $H' \subset G$. Now the robber can only move on $H'$ otherwise he will get caught by our stationary cop. Then $|H'|\leq m-\lfloor \sqrt{2m}\rfloor-2$, so $c_G(H)\leq 1+c_G(H')\leq 1+\left \lfloor{ \sqrt{2(m-\lfloor \sqrt{2m}\rfloor-2)}}\right \rfloor  \leq \lfloor \sqrt{2m}\rfloor$. (Note that this last inequality in best possible.) As this holds for any $H\subseteq G$ of order $m$, we get $c_G(m)\leq \lfloor \sqrt{2m}\rfloor$ as required. 
\end{proof}

In the case when $G$ is bipartite of diameter $3$, Lemma \ref{lem:degen} still holds (see Lemma 3 in \cite{lupeng}). As this was the only place where we used the diameter of $G$, all results go through identically for bipartite diameter $3$ graphs.

Let $G$ is a Moore graph of order $n$, that is, a strongly regular graph of diameter two, girth five and degree $\sqrt{n-1}$. Then by Theorem $3$ in \cite{aignerfromme} (stating that for graphs of girth five and minimum degree $d$ we have $c(G)\geq d$), and the fact that the neighbourhood of any vertex is a dominating set, we get $c(G)=\sqrt{n-1}$. While we could not find an infinite family of graphs with $c(G)\approx \sqrt{n}$, the following conjecture is likely to be close to best possible:

\begin{conj} \label{mainconj}If $G$ is a graph of diameter $2$, then $c(G)\leq \sqrt{n}$
\end{conj}

Moreover, we conjecture that the above holds for any graph $G$, not just for diameter $2$ graphs (hence the constant in Meyniel's conjecture should be $1$).

We now present an attempt at improving the bound given by Theorem \ref{main}. First note that the domination number of a diameter two graph may be as large as $C\sqrt{n\log n}$, see  \cite{domin_diam2}. Given a graph $G$, and a vertex $v$, we say that a cop \emph{controls} $v$ if the cop is on $v$ or on an adjacent vertex. Given a positive real $s$, say that $v$ is an \emph{$s$-trap} if one can place $\lfloor s \rfloor$ cops on the vertices of $G-\{v\}$ such that all neighbours of $v$ are controlled by the  cops. 

Note that if a graph has no $s$-trap then a robber may forever escape $\lfloor s \rfloor$ cops. So in order to have a chance at proving Conjecture \ref{mainconj}, it is necessary (but not sufficient) to establish that every graph has a $\sqrt{n}$-trap. We will use the following theorem proved by Chv\'{a}tal and McDiarmid \cite{chvatal} to do this:

\begin{thm}\label{chvatalthm}
Let $H$ be a $k$-uniform hypergraph (i.e. every edge contains $k$ vertices) with $n$ vertices and $m$ edges. Write $\tau (H)$ for the transversal of $H$, i.e. the size of a smallest set of vertices meeting all edges of $H$. Then

$$\tau (H) \leq \frac{\lfloor k/2 \rfloor m + n}{\lfloor 3k/2 \rfloor}$$

\end{thm}

\begin{lem}\label{rootntrap}
Every graph $G$ has a $\sqrt{n}$-trap. There are graphs with no $\left( \sqrt{n}-1\right)$-traps.
\end{lem}

\begin{proof}
For the second part, note that Moore graphs have no $\left(\sqrt{n}-1\right)$-traps, since every cop can control at most one neighbour of a fixed vertex $v$. For the first part, use induction on the order of the graph. If $|G|\leq 3$ we are done. Let $G$ have order $n\geq 4$. If $G$ has a vertex of degree at most $\lfloor \sqrt{n}\rfloor$ then we are done. If there exists a vertex of degree at least $2\lfloor \sqrt{n}\rfloor$ then put a stationary cop there, and we are done by induction since $1+\left \lfloor{ \sqrt{n-2\lfloor \sqrt{n}\rfloor-1}}\right \rfloor  \leq \lfloor \sqrt{n}\rfloor $. So we may assume every vertex has degree more than $\lfloor \sqrt{n} \rfloor$, and less than $2\lfloor \sqrt{n} \rfloor$.

Let $v$ be the vertex of minimal degree in $G$, say of degree $d$. Consider the hypergraph on vertex set $V(G)-\{v\}$, with edges $E_i$ for $i=1 \dots d$ being the neighbourhoods of the $d$ neighbours of $v$, excluding $v$ and including the vertex itself. Then $|E_i|\geq d$ for all $i$, and we want to show that this hypergraph has a transversal of size $\lfloor \sqrt{n}\rfloor$. It is sufficient to consider the case where $|E_i|=d$ for all $i$. Using Theorem \ref{chvatalthm} with $k=d$ and $m=d$, we conclude that $\tau (H) \leq \lfloor \sqrt{n}\rfloor$ and hence $v$ is a $\lfloor \sqrt{n}\rfloor$-trap.
\end{proof}

We have the following simple bound on the number of traps:

\begin{lem}\label{trapnumber}
Let $G$ be of order $n$ and let $\sqrt{n}\leq \alpha \leq n$. The number of $\alpha$-traps in $G$ is bigger than $\alpha - \sqrt{n-\alpha}- 1$.
\end{lem}
\begin{proof}
Induction on $n$, the claim holds for $n=1,2,3$. By Lemma \ref{rootntrap}, $G$ has a $\sqrt{n}$-trap, say $v$. Let $G' = G - \{v\}$. Note that if $u$ is an $(\alpha-1)$-trap in $G'$ then it is an $\alpha$-trap in $G$. If $\sqrt{n-1}\leq \alpha-1 \leq n-1$ then by induction the number of $(\alpha-1)$-traps in $G'$ is bigger than $\alpha-1 - \sqrt{n-\alpha}-1$ and the result follows. If, on the other hand, we have $\alpha-1<\sqrt{n-1}$, then $\alpha - \sqrt{n-\alpha}-1< 1$ and we are done since we have already established that $v$ is an $\alpha$-trap.
\end{proof}

Let $k=\lfloor \sqrt{n}\rfloor$ and let $\mathscr{P}=\mathscr{P}(G)$ be the graph whose vertices are the possible positions of the $k$ cops on the graph $G$, with $pq \in E(\mathscr{P})$ if and only if it is possible for the cops to move from position $p$ to position $q$, or vice versa. Then $\mathscr{P}$ is the $k$-fold strong product of $G$ with itself.

Consider the following relation $\preceq$ described by Clarke and MacGillivray in \cite{kcopwingraphs}. For $i=0,1,\dots$ we define $\preceq _i$ from $V(G)$ to $V(\mathscr{P})$ inductively as follows: 

\begin{enumerate}
\item For all $x\in V(G)$ and $p\in V(\mathscr{P})$, we have $x\preceq_0 p$ if, in position $p$, one of the cops is located at vertex $x$.
\item For $i>0$, we have $x\preceq_i p$ if, for every $y\in N_G(x)$, there exists $q\in N_\mathscr{P} (p)$ such that $y\preceq_j q$ for some $j<i$. 
\end{enumerate} If we have $\lfloor\sqrt{n}\rfloor$ cops and assume for simplicity that the robber is not allowed to pass, then Lemma \ref{rootntrap} is equivalent to the assertion that $\preceq _0$ is a proper subset of $\preceq _1$, and Lemma \ref{trapnumber} gives a lower bound on the size of $\preceq_1 \backslash \preceq_0$. The natural next step towards Meyniel's conjecture would be to prove that $\preceq_2$ is much bigger than $\preceq_1$, that is to say that in every graph that is not complete, there are many positions from which the cops need $2$ rounds to win -- but we could not prove this.

We now define a closely related game. The game of Teleporting Cops and Robbers differs from the usual game in that, now we allow the cops to jump to any vertex they want in their turn, except that they are not allowed to jump onto the robber. The robber loses if after his round he is in a neighbourhood of a cop. This makes the cops much stronger, especially in graphs of large diameter. We define the \emph{teleporting cop number} $c_T(G)$ to be the least number of teleporting cops required to catch the robber.

We note that clearly $c_T(G)\leq c(G)$ for any graph $G$. Moreover, Aigner and Fromme's proof (Theorem 3 in \cite{aignerfromme}) works for this game as well. Hence we instantly get that for Moore graphs we have $c_T(G)=c(G)=\sqrt{n-1}$, and for incidence graphs of projective planes we have $c_T(G)=c(G)\approx \sqrt{n/2}$ (see Theorem 4.1 in \cite{pralat}). The following theorem is an immediate corollary of Lemma \ref{rootntrap}.

\begin{cor}\label{telepcorol}
For every graph $G$, we have $c_T(G)\leq \sqrt{n}$.
\end{cor}
\begin{proof}
Let $v$ be a $\sqrt{n}$-trap. If the robber moves there at any point in the game, he gets caught immediately. So we may delete $v$ and we are done by induction.
\end{proof}

A natural question to ask is whether $c_T(G)=c(G)$ holds for every graph $G$. The answer is no --  a cubic graph of girth at least $80$ -- whose existence is proved in \cite{highgirthexists} --  has $c(G)>1000$, but $c_T(G)=3$  (see Theorem $1.1$ in \cite{frankl}). It would be interesting to further investigate how $c_T$ and $c$ are related. While the following conjecture seems too good to be true, we could not find a counterexample to it:

\begin{conj}\label{telepisnormal}
For any diameter $2$ graph $G$, we have $c(G)=c_T(G)$.
\end{conj}

We note that this, together with Corollary \ref{telepcorol},  would imply Conjecture \ref{mainconj}.

\end{document}